\documentclass[reqno,oneside,11pt]{amsart}
\usepackage{amsfonts,amssymb,amsthm,amsbsy,amscd,amsmath,bm,calc,circuitikz,color,dsfont,euscript, enumerate, geometry,graphicx,ifthen,latexsym,manfnt,marvosym,mathrsfs,marvosym,setspace,stmaryrd,textcomp,todonotes,tikz,tikz-cd, tcolorbox,upgreek,verbatim,xcolor}
\usepackage[utf8]{inputenc}
\usetikzlibrary{positioning,shapes,calc}
\definecolor{bulgarianrose}{rgb}{0.28, 0.02, 0.03}
\usepackage[linkcolor=bulgarianrose,citecolor=bulgarianrose,colorlinks=true,hypertexnames=false]{hyperref}
\allowdisplaybreaks

\newtheorem{theorem}{Theorem}[section]

\theoremstyle{definition}
\newtheorem*{definition}{Definition}
\newtheorem{construction}[theorem]{Construction}

\newtheorem*{remark}{Remark}

\newtheorem*{example}{Example}
\newtheorem*{acknowledgement}{Acknowledgement}
\numberwithin{equation}{section}

\newcommand{\PG}{\ensuremath{{\textnormal{PG}}}}

\makeatletter
\def\imod#1{\allowbreak\mkern10mu({\operator@font mod}\,\,#1)}
\def\@textbottom{\vskip\z@\@plus 18pt}
\let\@texttop\relax
\makeatother

\title[Upper bounds of the Tur\'{a}n density]{Upper bounds on the Tur\'{a}n density of projective hypergraphs over a finite field}
\author{Subhankar Dash}
\author{Kaushik Majumder}

\address{\newline 
\newline (a) School of Mathematical Sciences\\ \newline National Institute of Science Education and Research (NISER) Bhubaneswar\\\newline Jatni, Khurda-$752050$, Odisha, India.
\newline (b) Homi Bhabha National Institute (HBNI)\\ \newline Training School Complex, Anushakti Nagar, Mumbai- $400094$, India.
\newline \textnormal{\textestimated-Mails}: (Subhankar Dash) {\tt subhankar.dash\MVAt niser.ac.in, isubha.apple\MVAt gmail.com}
\newline \textnormal{\textestimated-Mails}: (Kaushik Majumder) {\tt kaushikmajumder\MVAt niser.ac.in, kaushikbnmajumder\MVAt gmail.com}}

\subjclass[2020]{Primary: 05C65, 05D05, 51E15.}
\keywords{$r-$graph, Tur\'{a}n density, projective plane and hyperplane, $\mathsf{L}^{p}-$method.}

\begin{document}

\begin{abstract}
For a prime power $q\geq2$ and an integer $m\geq2$, let $\PG(m,q)$ denote the $m-$dimensional projective hypergraph over the finite field $\mathbb{F}_q$. Employing the $\mathsf{L}^p-$method and using the incidence structure, we establish the Tur\'{a}n density of $\PG(m,q)$
\begin{equation*}
\pi(\PG(m,q))\leq\begin{cases}
  \left(1-\frac{1}{q+1}\right)^{\frac{1}{q}} & \text{ if } m=2,\\ 
  \left(1-\frac{1}{(q+1)|E(\PG(m-1,q))|}\right)^{\frac{1}{1+q^{m-1}}} & \text{ if } m\geq3.
\end{cases}   
\end{equation*}
The previously known upper bound on its Tur\'{a}n density was $1-\binom{q^m}{q}^{-1}$. As a consequence of our estimate, this upper bound is improved for each prime power $q\geq3$ and for each dimension index $m\geq2$. 

For a positive integer $t$, let $B_{t}[\PG(2,q)]$ be the $(q+1)-$graph formed by taking $t$ copies of the $\PG(2,q)$ that shares a common edge. Using the $\mathsf{L}^p-$method, we establish that the upper bound of the Tur\'{a}n density of $B_{t}[\PG(2,q)]$ is $\left(1-\frac{1}{q+1}\right)^{\frac{1}{qt}}$.
\end{abstract}

\maketitle

\section{Introduction}

One of the core problems in extremal combinatorics is to determine the Tur\'{a}n numbers. For an $r-$uniform hypergraph or simply $r-$graph $H$, the \emph{Tur\'{a}n number} of $H$ is the maximum number of $r-$edges in an $H-$free $r-$graph on $n$ vertices. We denote this number by $\mathrm{ex}(n,H)$. For such $H$, we note that the sequence $\left\{\binom{n}{r}^{-1}\mathrm{ex}(n,H)\right\}_{n=r}^{\infty}$ forms a monotone decreasing sequence in $[0,1]$. Hence it converges and has a limit, 
\begin{equation*}
\underset{n\to\infty}{\lim}\frac{\mathrm{ex}(n,H)}{\binom{n}{r}}=\inf\left\{\frac{\mathrm{ex}(n,H)}{\binom{n}{r}}:n\geq r\right\}.     
\end{equation*}
This limit is called the \emph{Tur\'{a}n density} of the $r-$graph $H$, denoted by $\pi(H)$.

An old conjecture of Vera T. S\'{o}s \cite{MR543051} asked whether Tur\'{a}n density of Fano plane free $3-$graph is $\frac{3}{4}$. Many authors answered affirmatively \cite{MR1750899,MR2160414,MR2176425} and established that, the Tur\'{a}n number
\begin{equation*}
\mathrm{ex}(n,\PG(2,2))=\binom{n}{3}-\binom{\lfloor n/2\rfloor}{3}-\binom{\lceil n/2\rceil}{3},
\end{equation*}
for sufficiently large integer values of $n$. These results and the incidence structures motivate the study of similar Turán problems for hypergraphs arising from finite projective geometries. It is established \cite[Theorem~1.1]{MR2156214} that for each dimension index $m\geq2$,
\begin{equation}\label{Keevash bound}
\overset{q}{\underset{t=1}\prod}\bigg(1-\frac{t}{q+q^{2}+\ldots+q^{m}}\bigg)\leq\pi(\PG(m,q))\leq1-\frac{1}{\binom{q^{m}}{q}}.    
\end{equation}
For $q=2$, the corresponding bounds on the Tur\'{a}n density \cite[Theorem~1.2]{MR2156214} are:
\begin{align*}
\frac{(2^{m+1}-3)(2^{m+1}-4)}{(2^{m+1}-2)^2}\leq\pi(\PG(m,2))\leq\begin{cases}
1-\frac{3}{2^{2m}-1}&\quad\textnormal{ if }m\textnormal{ is odd},\\
1-\frac{6}{(2^{m}-1)(2^{m+1}+1)}&\quad\textnormal{ if }m\textnormal{ is even}.
\end{cases}
\end{align*}
In this article, Theorem~\ref{Turan density of projective plane} and Theorem~\ref{Turan density of projective hyperplane} shows that: 
\begin{equation*}
\pi(\PG(m,q))\leq\begin{cases}
  \bigg(1-\frac{1}{q+1}\bigg)^{\frac{1}{q}} & \text{ if } m=2,\\ 
  \bigg(1-\frac{(q-1)^{2}}{(q^{m}-1)(q^{m-1}-1)}\bigg)^{\frac{1}{1+q^{m-1}}} & \text{ if } m\geq3
\end{cases}   
\end{equation*}
which improves the upper bound~\cite{MR2156214} for each prime power $q\geq3$ and each integer $m\geq2$. For the best known lower bound of $\pi(\PG(m,q))$ we refer \cite[Theorem~1.2]{MR4548958}.

We establish our results in two different sections. In section~3 we deal with the case of $m=2$, and $m\geq3$ in the next section. This is due to the fact that $E(\PG(m,q))$ forms an intersecting family iff $m=2$. For $\PG(2,q)$, the edges adjacent to the vertices of the edge at infinity cover all other edges. So, fixing one edge, we write our local forbidden constraint. This yields a better result. However, for $m\geq3$, the same technique does not work. This forces us to write our local forbidden constraint by summing over the edges at infinity.
 
Our proof proceeds in two main stages. First, we use the incidence structure of the \\
$m-$dimensional projective hypergraph over the $\mathbb{F}_{q}$. Here we write the counting problem as a suitable combination of sum and products of indicator functions to express the structure in its analytic form and characterize the forbidden properties. This gives rise to a \emph{local forbidden constraint}. This portion is the main non-trivial key step of this article. The local forbidden constraint implies the \emph{global forbidden constraint}. These terminologies will make more sense in section~3 and section~4. In the second stage, we put the global forbidden constraint to a form by interchanging summations (discrete version of the Fubini's theorem). Such interchange allows us to apply the following version of the H\"{o}lder's inequality:- for $p\in[1,\infty)$ and $a_{i}\geq0$,
\begin{equation*}
\underset{i\in A}{\sum}a_{i}^{p}\geq\frac{1}{|A|^{p-1}}\bigg(\underset{i\in A}{\sum}a_{i}\bigg)^{p},
\end{equation*}
where $A$ is a non-empty finite index set. The structural properties of projective hypergraph combined with suitable use of the Fubini's theorem and H\"{o}lder's inequality, yield the required extremal estimate.

Motivated from \cite{MR2110982, MR137661}, we consider the $(q+1)-$graph analogue of the \emph{book graph} involving the copies of the projective plane $\PG(2,q)$ as pages for further application of the $\mathsf{L}^p-$method. The primary motivation for considering $\PG(2,q)$ is that it has an equal number of points and lines, making it analogue of the cycle graph. Its rich incidence structure and the symmetry between points and lines make it a suitable building block for a blow-up construction, in which multiple copies of $\PG(2,q)$ proliferate around a common line, forming the pages of a book.

Informally, for a positive integer $t$ and an $r-$graph $H$, the \emph{$t-$page book with pages $H$}, denoted by $B_{t}[H]$, is the $r-$graph obtained from $t$ mutually disjoint copies of $H$ by identifying one $r-$edge from each copy into a single common edge. Such common edge is called as the \emph{spine} of the book. We note here that, $B_{1}[H]=H$. The term ``identification'' comes from the familiar geometric construction in which two parallel edges of a rectangle are identified (or glued together) to form a cylinder. For the formal definition and construction of $B_{t}[H]$, we refer the reader to the last section of this article, where these notions are presented in an abstract setting. In the last section of this article, we study the upper bound of Tur\'{a}n density of $B_{t}[\PG(2,q)]$. In Theorem~\ref{Turan density of projective book}, we establish that
\begin{equation*}
\pi(B_{t}[\PG(2,q)])\leq\left(\frac{q}{q+1}\right)^{\frac{1}{qt}}.
\end{equation*}

We refer this proof methodological framework as the \emph{$\mathsf{L}^{p}$-method}. It is inspired by the proof techniques developed in \cite{MR2195580} and the proof of the Mantel's theorem in \cite{MR4603631}. Applying this method, we also give an alternative proof of the classical upper bound of Erd\"{o}s result for complete $r$-partite $r$-graphs in \cite{DM}. The method may also be useful for other extremal problems involving multilinear counting expressions.

\section{Preliminaries}

In this section, we define the terminologies used in this article. We begin by defining the $m-$dimensional projective hypergraph $\PG(m,q)$ as given in \cite{MR2156214}.
\begin{definition}
Let $\mathbb{F}_q$ denote the finite field with $q$ elements, where $q\geq2$ is a prime power. For an integer $m\geq2$ the $m-$dimensional projective hypergraph $\PG(m,q)$ over $\mathbb{F}_{q}$ is defined as follows. Its vertex set consists of all the one dimensional subspaces of $\mathbb{F}_q^{m+1}$ and each edge set is the set of one dimensional subspaces of a two dimensional subspace. That is,
\begin{align*}
V(\PG(m,q))&=\{\langle u\rangle:u\in \mathbb{F}_q^{m+1}\}\\
E(\PG(m,q))&=\{E_{W}:W\text{ is a $2-$dimensional subspace of }\mathbb{F}_q^{m+1}\},
\end{align*}
where $E_{W}=\{\langle u\rangle:u\in W\smallsetminus\{0\}\}$ for each $2-$dimensional subspace $W$ of $\mathbb{F}_q^{m+1}$.
\end{definition}

\begin{example}
The projective plane $\PG(2,2)$ is popularly known as Fano Plane. Its vertex set and edge set is,
\begin{align*}
V(\PG(2,2))&=\bigg\{\langle (1,0,0)\rangle,\langle (0,1,0)\rangle,\langle (0,0,1)\rangle,\langle (1,1,0)\rangle,\langle (0,1,1)\rangle,\langle (1,0,1)\rangle,\langle (1,1,1)\rangle\bigg\}\\
E(\PG(2,2))&=\bigg\{\left\{\langle(1,0,0)\rangle,\langle(0,1,0)\rangle,\langle(1,1,0)\rangle\right\},\left\{\langle(1,0,0)\rangle,\langle(0,0,1)\rangle,\langle(1,0,1)\rangle\right\},\\
&\hspace{0.85cm}\left\{\langle(0,1,0)\rangle,\langle(0,0,1)\rangle,\langle(0,1,1)\rangle\right\},\left\{\langle(1,0,0)\rangle,\langle(0,1,1)\rangle,\langle(1,1,1)\rangle\right\},\\
&\hspace{0.85cm}\left\{\langle(0,1,0)\rangle,\langle(1,0,1)\rangle,\langle(1,1,1)\rangle\right\},\left\{\langle(0,0,1)\rangle,\langle(1,1,0)\rangle,\langle(1,1,1)\rangle\right\},\\
&\hspace{0.85cm}\left\{\langle(1,1,0)\rangle,\langle(0,1,1)\rangle,\langle(1,0,1)\rangle\right\}\bigg\}.
\end{align*}
For notational simplicity, we denote  $V(\PG(2,2))=\{x,y,z,a,b,c,d\}$ and 
\begin{align*}
E(\PG(2,2)=\left\{\{x,y,z\},
\begin{aligned}
&\{x,a,b\},\{y,a,d\},\{z,a,c\},\\
&\{x,c,d\},\{y,b,c\},\{z,b,d\}  
\end{aligned}
\right\}.    
\end{align*}
\end{example}

Using the following construction, we write down the structure of $\PG(m,q)$ explicitly.

\begin{construction}
Let $m\ge2$ and let $\mathbb{F}_q$ be the finite field with $q-$elements for a prime power integer $q\geq2$. Associate a symbol $x_{a}$ for each $a\in\mathbb{F}_q^m$ and construct
\begin{equation*}
L_{a,u}=\{x_{a+tu}:t\in\mathbb{F}_q\}.
\end{equation*}
for each $a\in\mathbb{F}_q^m$ and $u\in\mathbb{F}_q^m\setminus\{0\}$. Define the direction of $L_{a,u}$ as $\text{dir}(L_{a,u})=\{\langle u\rangle\}$.

We construct the hypergraph $H(m,q)$ with vertex and edge sets
\begin{align*}
V(H(m,q))&=\{x_{a}:a\in\mathbb{F}_q^m\}\sqcup\{\langle u\rangle:u\in\mathbb{F}_q^m\setminus\{0\}\}\\
E(H(m,q))&=\left\{L_{a,u}\sqcup\{\langle u\rangle\}:a\in\mathbb{F}_q^m,\ u\in\mathbb{F}_q^m\setminus\{0\}\right\}\sqcup\\
&\hspace{2cm}\sqcup\left\{E_W:W\text{ is a two dimensional subspace of }\mathbb{F}_q^m\right\}.
\end{align*}
Edges of the form $L_{a,u}\sqcup\{\langle u\rangle\}$ are called \emph{extended affine edges}. Edges of the form $E_{W}$ are called \emph{edges at infinity}. 
\end{construction}

\begin{theorem}
The $(q+1)-$hypergraph $H(m,q)$ is isomorphic to $\PG(m,q)$.
\end{theorem}
\begin{proof}
Define $\phi:V(H(m,q))\longrightarrow V(\PG(m,q))$ by 
\begin{align*}
\phi(x_{a_1,\ldots,a_m})&=\langle(a_1,\ldots,a_m,1)\rangle,\forall a_1,\cdots,a_m\in\mathbb{F}_q\\
\phi(\langle (u_1,\cdots,u_m)\rangle)&=\langle(u_1,\ldots,u_m,0)\rangle,\forall(u_{1},\cdots,u_{m})\in\mathbb{F}_q^m\setminus\{0\}.
\end{align*}
Then $\phi$ forms the required isomorphism between $H(m,q)$ and $\PG(m,q)$.
\end{proof}

We have the following results that follows from the incidence properties of $\PG(m,q)$.
\begin{theorem}
\label{properties of projective plane}
The $m-$dimensional projective geometry over $\mathbb{F}_q$ i.e. $\PG(m,q)$ satisfy the following properties .
\begin{enumerate}[\normalfont(a)]
\item $|V(\PG(m,q))|=\frac{q^{m+1}-1}{q-1}=1+q+\cdots+q^m$,
\item $|E(\PG(m,q))|=\frac{(q^{m+1}-1)(q^{m}-1)}{(q^{2}-1)(q-1)}$,
\item $\PG(m,q)$ is $q+1$ uniform i.e. each hyperedge has size $q+1$,
\item $\PG(m,q)$ is regular, in the sense that each vertex contained in exactly 
\begin{equation*}
 \frac{q^{m}-1}{q-1}=1+q+\cdots+q^{m-1}.   
\end{equation*}
many hyperedges.
\item $E(\PG(m,q))$ forms an intersecting family iff $m=2$.
\end{enumerate}
\end{theorem}

\begin{construction}
Let $u\in\mathbb{F}_q^m\smallsetminus\{0\}$. Construct an equivalence relation $\sim_u$ on the affine vertices $\left\{x_{a}:a\in\mathbb{F}_q^m\right\}$ as follows. 
\begin{equation*}
x_{a_1,\ldots,a_m}\sim_u x_{b_1,\ldots,b_m}\text{if and only if }b_i-a_i=tu_i,\forall i\in[m]\text{ for some }t\in\mathbb{F}_q.
\end{equation*}
Then $[x_{a_1,\ldots,a_m}]_{\sim_u}=L_{a,u}$ where $a=(a_1,\ldots,a_m)$.

Again construct another equivalence relaton $\parallel$ on the set of extended affine edges \\ $\{L_{a,u}\sqcup\{\langle u\rangle\}:a\in\mathbb{F}_q^m,\ u\in\mathbb{F}_q^m\setminus\{0\}\}$ as follows.
\begin{equation*}
L_{a,u}\sqcup\{\langle u\rangle\}\parallel L_{b,v}\sqcup\{\langle v\rangle\}\text{ if and only if }\text{dir}(L_{a,u}\sqcup\{\langle u\rangle\})=\text{dir}(L_{b,v}\sqcup\{\langle v\rangle\})\text{ i.e. }\langle u\rangle=\langle v\rangle.
\end{equation*}
We call $P_{\langle u\rangle}=[L_{a,u}\sqcup\{\langle u\rangle\}]_{\parallel}$ a parallel class with direction $\langle u\rangle$. Then we have a collection of parallel classes
\begin{equation*}
S(x_a:a\in\mathbb{F}_q^m)=\left\{P_{\langle u\rangle}(x_a:a\in\mathbb{F}_q^m):u\in\mathbb{F}_q^m\smallsetminus\{0\}\right\}.
\end{equation*}
\end{construction}

\begin{remark}
Each $[L_{a,u}]_{\sim_u}$ partitions the set of the affine vertices, however the collection of parallel classes partitions the set of extended affine edges. Hence each parallel class has size $q^{m-1}$ and there are total $\frac{q^m-1}{q-1}=1+q+\cdots+q^{m-1}$ many parallel classes.
\end{remark}

\begin{example}
We list down the parallel classes in case of $\PG(3,2)$ below. The vertex set is
\begin{equation*}
 V(\PG(3,2))=\left\{x_a:a\in\mathbb{F}_2^3\right\}\sqcup\left\{x,y,z,\alpha,\beta,\gamma,\delta\right\}.
\end{equation*}
For notational simplicity, we write it as follows.
\begin{equation*}
 V(\PG(3,2))=\left\{x_a:a\in[8]\right\}\sqcup\left\{x,y,z,\alpha,\beta,\gamma,\delta\right\}.
\end{equation*}
The edge set is, 
\begin{align*}
E(\PG(3,2))&=\underbrace{\left\{\begin{aligned}
 &\{x,x_{1},x_{2}\},\{x,x_{3},x_{4}\},\{x,x_{5},x_{6}\},\{x,x_{7},x_{8}\},\\
&\{y,x_1,x_3\},\{y,x_2,x_4\},\{y,x_5,x_7\},\{y,x_6,x_8\},\\
&\{z,x_1,x_5\},\{z,x_2,x_6\},\{z,x_3,x_7\},\{z,x_4,x_8\},\\
&\{\alpha,x_1,x_8\},\{\alpha,x_2,x_7\},\{\alpha,x_3,x_6\},\{\alpha,x_4,x_5\},\\
&\{\beta,x_1,x_4\},\{\beta,x_2,x_3\},\{\beta,x_5,x_8\},\{\beta,x_6,x_7\},\\
&\{\gamma,x_1,x_7\},\{\gamma,x_2,x_8\},\{\gamma,x_3,x_5\},\{\gamma,x_4,x_6\},\\
&\{\delta,x_1,x_6\},\{\delta,x_2,x_5\},\{\delta,x_3,x_8\},\{\delta,x_4,x_7\}\\
\end{aligned}\right\}}_{\textnormal{extended affine edges}}\\
&\hspace{2cm}\bigsqcup\underbrace{\left\{\{x,y,z\},
\begin{aligned}
&\{x,\alpha,\beta\},\{y,\alpha,\delta\},\{z,\alpha,\gamma\},\\
&\{x,\gamma,\delta\},\{y,\beta,\gamma\},\{z,\beta,\delta\}  
\end{aligned}
\right\}}_{\textnormal{edges at infinity}}.       
\end{align*}

The set of parallel classes is
\begin{align*}
S(x_a:a\in[8])=\left\{\begin{aligned}
 &P_x=\left\{\{x_{1},x_{2}\},\{x_{3},x_{4}\},\{x_{5},x_{6}\},\{x_{7},x_{8}\}\right\},\\
&P_y=\left\{\{x_1,x_3\},\{x_2,x_4\},\{x_5,x_7\},\{x_6,x_8\}\right\},\\
&P_z=\left\{\{x_1,x_5\},\{x_2,x_6\},\{x_3,x_7\},\{x_4,x_8\}\right\},\\
&P_{\alpha}=\left\{\{x_1,x_8\},\{x_2,x_7\},\{x_3,x_6\},\{x_4,x_5\}\right\},\\
&P_{\beta}=\left\{\{x_1,x_4\},\{x_2,x_3\},\{x_5,x_8\},\{x_6,x_7\}\right\},\\
&P_{\gamma}=\left\{\{x_1,x_7\},\{x_2,x_8\},\{x_3,x_5\},\{x_4,x_6\}\right\},\\
&P_{\delta}=\left\{\{x_1,x_6\},\{x_2,x_5\},\{x_3,x_8\},\{x_4,x_7\}\right\}
\end{aligned}\right\}.       
\end{align*}
Here $P_{y}=P_{y}(x_a:a\in[8])$ means the set of parallel classes with direction $y$ and so on.
\end{example}

\section{Tur\'{a}n density estimation of \texorpdfstring{$\PG(2,q)$}{PG(2,q)}} 
Before we estimate the Tur\'{a}n density of $\PG(2,q)$, we discuss a small case for $q=2$ first i.e., the Fano plane $\PG(2,2)$.
\begin{theorem}\label{PG(2,2)}
Each $3-$graph $H$ on $n$ vertices that does not contain a Fano plane has at most $\sqrt{\frac{2}{3}}\frac{n^3}{3!}-$many edges, i.e 
\begin{equation*}
\mathrm{ex}(n,\PG(2,2))\leq\sqrt{\frac{2}{3}}\frac{n^3}{3!}.
\end{equation*}
Moreover, the Tur\'{a}n density $\pi(\PG(2,2))\leq\sqrt{\frac{2}{3}}$.
\end{theorem}
\begin{proof}
Let $H$ be a $\PG(2,2)$(Fano plane)-free $3-$graph on $n-$vertices. Let $\{x,y,z\}\in E(H)$. Then for each $a,b,c,d\in V(H)$, we have the following local forbidden constraint.
\begin{align}\label{local PG(2,2)}
&\mathds{1}_{E(H)}(\{x,a,b\})\mathds{1}_{E(H)}(\{x,c,d\})+\mathds{1}_{E(H)}(\{y,a,c\})\mathds{1}_{E(H)}(\{y,b,d\})\notag\\
&\hspace{2cm}+\mathds{1}_{E(H)}(\{z,a,d\})\mathds{1}_{E(H)}(\{z,b,c\})\leq2.      
\end{align}
To see this, suppose (if possible),
\begin{align*}
&\mathds{1}_{E(H)}(\{x,a,b\})\mathds{1}_{E(H)}(\{x,c,d\})+\mathds{1}_{E(H)}(\{y,a,c\})\mathds{1}_{E(H)}(\{y,b,d\})\\
&\hspace{2cm}+\mathds{1}_{E(H)}(\{z,a,d\})\mathds{1}_{E(H)}(\{z,b,c\})=3.
\end{align*}
Then 
\begin{align*}
&\mathds{1}_{E(H)}(\{x,a,b\})=\mathds{1}_{E(H)}(\{x,c,d\})=\mathds{1}_{E(H)}(\{y,a,c\})=\mathds{1}_{E(H)}(\{y,b,d\})\\
&\hspace{2cm}=\mathds{1}_{E(H)}(\{z,a,d\})=\mathds{1}_{E(H)}(\{z,b,c\})=1.
\end{align*}
This implies that $x,y,z,a,b,c,d$ are all distinct with 
\begin{equation*}
\bigg\{\{x,a,b\},\{x,c,d\},\{y,a,c\},\{y,b,d\},\{z,a,d\},\{z,b,c\}\bigg\}\subset E(H).   
\end{equation*}
Together with $\{x,y,z\}\in E(H)$, this forms a copy of $\PG(2,2)$ in $H$.
A contradiction arises, since $H$ does not contain a copy of Fano plane. This establishes inequality~\eqref{local PG(2,2)}. Hence,
\begin{align}\label{global PG(2,2)}
\underset{\{x,y,z\}\in E(H)}{\sum}\underset{a,b,c,d\in V(H)}{\sum}
&\bigg(\mathds{1}_{E(H)}(\{x,a,b\})
\mathds{1}_{E(H)}(\{x,c,d\})+\mathds{1}_{E(H)}(\{y,a,c\})
\mathds{1}_{E(H)}(\{y,b,d\})+\notag\\
&\quad+\mathds{1}_{E(H)}(\{z,a,d\})
\mathds{1}_{E(H)}(\{z,b,c\})\bigg)
\leq 2n^4|E(H)|.
\end{align}
We refer to inequality~\eqref{global PG(2,2)} as the \emph{global forbidden constraint} for $\PG(2,2)$.

We now derive a lower bound for the left-hand side term of the inequality~\eqref{global PG(2,2)}. By symmetry, it suffices to consider the first term, as the second and third terms can be handled in the same way. Here, we apply H\"{o}lder's inequality at the last step.
\begin{align*}
&\underset{\{x,y,z\}\in E(H)}{\sum}\underset{a,b,c,d\in V(H)}{\sum}\mathds{1}_{E(H)}(\{x,a,b\})\mathds{1}_{E(H)}(\{x,c,d\})\\
=&\frac{1}{6}\underset{x,y,z\in V(H)}{\sum}\mathds{1}_{E(H)}(\{x,y,z\})\left(\underset{a,b\in V(H)}{\sum}\mathds{1}_{E(H)}(\{x,a,b\})\right)^2
=\frac{1}{6}\underset{x\in V(H)}{\sum}\left(\underset{a,b\in V(H)}{\sum}\mathds{1}_{E(H)}(\{x,a,b\})\right)^{3}\\
\geq&\frac{1}{6}\frac{1}{n^2}\left(\underset{x,a,b\in V(H)}{\sum}\mathds{1}_{E(H)}(\{x,a,b\})\right)^3=\frac{1}{6n^2}\left(3!|E(H)|\right)^3.
\end{align*}

Therefore,
\begin{tcolorbox}
\begin{equation*}
 3\frac{1}{6n^{2}}(3!|E(H)|)^3\leq2n^4|E(H)|   
\end{equation*}
\end{tcolorbox}
solving which we obtain the required result.
\end{proof}

\begin{remark}
The established bound is far from the actual Tur\'{a}n density of the Fano plane. But the above $\mathsf{L}^{p}-$method of deduction can be generalized for any prime power $q\geq3$. 
\end{remark}

\begin{theorem}\label{Turan density of projective plane}
Each $(q+1)-$graph $H$ on $n$ vertices that does not contain $(q+1)-$graph $\PG(2,q)$, contains at most  $\left(1-\frac{1}{q+1}\right)^{\frac{1}{q}}\frac{n^{q+1}}{(q+1)!}-$many edges, i.e 
\begin{equation*}
\mathrm{ex}(n,\PG(2,q))\leq\left(1-\frac{1}{q+1}\right)^{\frac{1}{q}}\frac{n^{q+1}}{(q+1)!}.   
\end{equation*}
Moreover, $\pi(\PG(2,q))\leq\left(1-\frac{1}{q+1}\right)^{\frac{1}{q}}$.
\end{theorem}
\begin{proof}
Let $\PG(2,q)$ be the two dimensional projective geometry over $\mathbb{F}_q$ and $H$ be a $\PG(2,q)-$free $(q+1)-$graph on $n-$vertices. Let $\{y_1,\cdots,y_{q+1}\}\in E(H)$ and $\{x_a:a\in\mathbb{F}_q^2\}\subseteq V(H)$. As $|S(x_a:a\in\mathbb{F}_q^2)|=q+1$, without loss of generality we denote 
\begin{equation*}
S(x_a:a\in\mathbb{F}_q^2)=\bigg\{P_1(x_a:a\in\mathbb{F}_q^2),\cdots,P_{q+1}(x_a:a\in\mathbb{F}_q^2)\bigg\}.    
\end{equation*}
Since $H$ is $\PG(2,q)-$free, the local forbidden constraint is
\begin{equation}\label{local PG(2,q)}
\underset{m\in[q+1]}{\sum}\underset{e\in P_m(x_a:a\in\mathbb{F}_q^2)}{\prod}\mathds{1}_{E(H)}(\{y_{m}\}\sqcup e)\leq q.
\end{equation}
If not, using similar arguments used to establish inequality~\eqref{local PG(2,2)}, $\{y_1,\cdots,y_{q+1}\}$ together with $\{x_a:a\in\mathbb{F}_q^2\}$ forms a copy of $\PG(2,q)$, which leads to a contradiction.

Hence, summing~\eqref{local PG(2,q)} over each edge $\{y_1,\ldots,y_{q+1}\}\in E(H)$ and each $\{x_a:a\in\mathbb{F}_q^2\}\subseteq V(H)$, we obtain the following global forbidden constraint.
\begin{equation}\label{global PG(2,q)}
\begin{aligned}
\underset{\{y_1,\cdots,y_{q+1}\}\in E(H)}{\sum}\underset{\{x_a:a\in\mathbb{F}_q^2\}\subseteq V(H)}{\sum}\underset{m\in[q+1]}{\sum}\underset{e\in P_m(x_a:a\in\mathbb{F}_q^2)}{\prod}\mathds{1}_{E(H)}(\{y_{m}\}\sqcup e)\leq qn^{q^2}|E(H)|.     
\end{aligned}   
\end{equation}

We now derive a lower bound for the left-hand side term of the inequality~\eqref{global PG(2,q)}. We calculate the lower bound for the first term $m=1$. The other terms can be estimated in the same way. While estimating the following expression, we use the structure of the parallel classes i.e. each parallel class partitions the set $\{x_a:a\in\mathbb{F}_q^2\}$ in the second line and H\"{o}lder's inequality in the last line.
\begin{align*}
&\underset{\{y_1,\cdots,y_{q+1}\}\in E(H)}{\sum}\underset{\{x_a:a\in\mathbb{F}_q^2\}\subseteq V(H)}{\sum}\bigg(\underset{e\in P_1(x_a:a\in\mathbb{F}_q^2)}{\prod}\mathds{1}_{E(H)}(\{y_{1}\}\sqcup e)\bigg)\\
=&\frac{1}{(q+1)!}\underset{y_1,\cdots,y_{q+1}\in V(H)}{\sum}\mathds{1}_{E(H)}(\{y_1,\cdots,y_{q+1}\})\bigg(\underset{x_1,\cdots,x_q\in V(H)}{\sum}\mathds{1}_{E(H)}(\{y_{1},x_1,\cdots,x_q\}\bigg)^q\\
=&\frac{1}{(q+1)!}\underset{y_1\in V(H)}{\sum}\bigg(\underset{y_2,\cdots,y_{q+1}\in V(H)}{\sum}\mathds{1}_{E(H)}(\{y_{1},\cdots,y_{q+1}\}\bigg)^{q+1}\\
\geq&\frac{1}{(q+1)!}\frac{1}{n^q}\bigg(\underset{y_1,\cdots,y_{q+1}\in V(H)}{\sum}\mathds{1}_{E(H)}(\{y_{1},\cdots,y_{q+1}\}\bigg)^{q+1}
=\frac{1}{(q+1)!}\frac{1}{n^q}\bigg((q+1)!|E(H)|\bigg)^{q+1}.
\end{align*}
Hence
\begin{align*}
&\underset{\{y_1,\cdots,y_{q+1}\}\in E(H)}{\sum}\underset{\{x_a:a\in\mathbb{F}_q^2\}\subseteq V(H)}{\sum}\underset{m\in[q+1]}{\sum}\underset{e\in P_m(x_a:a\in\mathbb{F}_q^2)}{\prod}\mathds{1}_{E(H)}(\{y_{m}\}\sqcup e)\\
&\hspace{2cm}\geq(q+1)\frac{1}{(q+1)!}\frac{1}{n^q}\bigg((q+1)!|E(H)|\bigg)^{q+1}.
\end{align*}

Therefore,
\begin{tcolorbox}
\begin{equation*}
(q+1)\frac{1}{(q+1)!}\frac{1}{n^q}\bigg((q+1)!|E(H)|\bigg)^{q+1}\leq qn^{q^2}|E(H)|    
\end{equation*}
\end{tcolorbox}
solving which we obtain the required result.
\end{proof}

\section{Tur\'{a}n density estimation of \texorpdfstring{$\PG(m,q),m\geq3$}{PG(m,q)}}

In this section, we generalize Theorem~\eqref{Turan density of projective plane} and calculate an upper bound of the Tur\'{a}n density of $\PG(m,q)$, for $m\geq3$. Before that, we discuss the case of $m=3$ and $q=2$ for better understanding of the generalization of $m=2$ to $m=3$.

\begin{theorem}
Each $3-$graph $H$ on $n$ vertices that does not contain a copy of $\PG(3,2)$ has at most $\left(\frac{20}{21}\right)^{\frac{1}{5}}\frac{n^3}{3!}-$many edges, i.e.
\begin{equation*}
\mathrm{ex}(n,\PG(3,2))\leq\left(1-\frac{1}{3|E(\PG(2,2))|}\right)^{\frac{1}{5}}\frac{n^3}{3!}.
\end{equation*}
Moreover, the Tur\'{a}n density $\pi(\PG(3,2))\leq\left(\frac{20}{21}\right)^{\frac{1}{5}}$.
\end{theorem}

\begin{proof}Let $H$ be $\PG(3,2)-$free. Then for each $x_{1},x_{2},x_{3},x_{4},x_{5},x_{6},x_{7},x_{8},x,y,z,\alpha,\beta,\gamma,\delta\in V(H)$ we have the following local forbidden constraint.
\begin{align}\label{local PG(3,2)}
\underset{\substack{E\in E(\PG(2,2))\\V(\PG(2,2))=\{x,y,z,\alpha,\beta,\gamma,\delta\}}}{\sum}\mathds{1}_{E(H)}(E)
&\bigg[\underset{u\in E}{\sum}\underset{A\in P_u(x_{i}:i\in[8])}{\prod}\mathds{1}_{E(H)}(\{u\}\sqcup A)\bigg]\leq3|E(\PG(2,2))|-1.
\end{align}
This implies the following global forbidden constraint.
\begin{align}\label{global PG(3,2)}
&\underset{\substack{x_{1},x_{2},x_{3},x_{4},x_{5},x_{6},x_{7},x_{8}\in V(H)\\x,y,z,\alpha,\beta,\gamma,\delta\in V(H)}}{\sum}\underset{\substack{E\in E(\PG(2,2))\\V(\PG(2,2))=\{x,y,z,\alpha,\beta,\gamma,\delta\}}}{\sum}\mathds{1}_{E(H)}(E)
\bigg[\underset{u\in E}{\sum}\underset{A\in P_u(x_{i}:i\in[8])}{\prod}\mathds{1}_{E(H)}(\{u\}\sqcup A)\bigg]\notag\\
&\hspace{7cm}\leq(3|E(\PG(2,2))|-1)n^{15}.
\end{align}
We calculate a lower bound below using H\"{o}lder's inequality for one edge at infinity. 
\begin{align*}
&\underset{\substack{x_{1},x_{2},x_{3},x_{4},x_{5},x_{6},x_{7},x_{8}\in V(H)\\x,y,z,\alpha,\beta,\gamma,\delta\in V(H)}}{\sum}\mathds{1}_{E(H)}(\{x,y,z\})\bigg[\underset{A\in P_x(x_{i}:i\in[8])}{\prod}\mathds{1}_{E(H)}(\{x\}\sqcup A)\\
&\hspace{4em}+\underset{A\in P_y(x_{i}:i\in[8])}{\prod}\mathds{1}_{E(H)}(\{y\}\sqcup A)+\underset{A\in P_z(x_{i}:i\in[8])}{\prod}\mathds{1}_{E(H)}(\{z\}\sqcup A)\bigg]\\
=&n^4\underset{\substack{x_{1},x_{2},x_{3},x_{4},x_{5},x_{6},x_{7},x_{8}\in V(H)\\x,y,z\in V(H)}}{\sum}\mathds{1}_{E(H)}(\{x,y,z\})\bigg[\underset{A\in P_x(x_{i}:i\in[8])}{\prod}\mathds{1}_{E(H)}(\{x\}\sqcup A)\\
&\hspace{4em}+\underset{A\in P_y(x_{i}:i\in[8])}{\prod}\mathds{1}_{E(H)}(\{y\}\sqcup A)+\underset{A\in P_z(x_{i}:i\in[8])}{\prod}\mathds{1}_{E(H)}(\{z\}\sqcup A)\bigg]\\
\end{align*}
We shall calculate the first term.
\begin{align*}
&\underset{\substack{x_{1},x_{2},x_{3},x_{4},x_{5},x_{6},x_{7},x_{8}\in V(H)\\x,y,z\in V(H)}}{\sum}\mathds{1}_{E(H)}(\{x,y,z\})\underset{A\in P_x(x_{i}:i\in[8])}{\prod}\mathds{1}_{E(H)}(\{x\}\sqcup A)\\
=&\underset{\substack{x_{1},x_{2},x_{3},x_{4},x_{5},x_{6},x_{7},x_{8}\in V(H)\\x,y,z\in V(H)}}{\sum}\mathds{1}_{E(H)}(\{x,y,z\})\mathds{1}_{E(H)}(\{x,x_{1},x_{2}\})\mathds{1}_{E(H)}(\{x,x_{3},x_{4}\})\\
&\hspace{7cm}\mathds{1}_{E(H)}(\{x,x_{5},x_{6}\})\mathds{1}_{E(H)}(\{x,x_{7},x_{8}\})\\
=&\underset{x\in V(H)}{\sum}\left(\underset{y,z\in V(H)}{\sum}\mathds{1}_{E(H)}(\{x,y,z\})\right)^5\\
\geq&\frac{1}{n^4}\left(\underset{x,y,z\in V(H)}{\sum}\mathds{1}_{E(H)}(\{x,y,z\})\right)^5
=\frac{1}{n^4}\left(3!|E(H)|\right)^5.
\end{align*}
Then
\begin{align*}
&\underset{\substack{x_{1},x_{2},x_{3},x_{4},x_{5},x_{6},x_{7},x_{8}\in V(H)\\x,y,z,\alpha,\beta,\gamma,\delta\in V(H)}}{\sum}\mathds{1}_{E(H)}(\{x,y,z\})\bigg[\underset{A\in P_x(x_{i}:i\in[8])}{\prod}\mathds{1}_{E(H)}(\{x\}\sqcup A)\\
&\hspace{4em}+\underset{A\in P_y(x_{i}:i\in[8])}{\prod}\mathds{1}_{E(H)}(\{y\}\sqcup A)+\underset{A\in P_z(x_{i}:i\in[8])}{\prod}\mathds{1}_{E(H)}(\{z\}\sqcup A)\bigg]\\
&\hspace{4cm}\geq n^{4}3\frac{1}{n^4}\left(3!|E(H)|\right)^5=3\left(3!|E(H)|\right)^5.
\end{align*}

Hence,
\begin{align*}
&\underset{\substack{x_{1},x_{2},x_{3},x_{4},x_{5},x_{6},x_{7},x_{8}\in V(H)\\x,y,z,\alpha,\beta,\gamma,\delta\in V(H)}}{\sum}\underset{\substack{E\in E(\PG(2,2))\\V(\PG(2,2))=\{x,y,z,\alpha,\beta,\gamma,\delta\}}}{\sum}\mathds{1}_{E(H)}(E)
\bigg[\underset{u\in E}{\sum}\underset{A\in P_u(x_{i}:i\in[8])}{\prod}\mathds{1}_{E(H)}(\{u\}\sqcup A)\bigg]\\
&\hspace{4cm}\geq|E(\PG(2,2))|3\left(3!|E(H)|\right)^5.
\end{align*}

Therefore,
\begin{tcolorbox}
\begin{equation*}
3|E(\PG(2,2))|\left(3!|E(H)|\right)^5\leq(3|E(\PG(2,2))|-1)n^{15}
\end{equation*}
\end{tcolorbox}
solving which we obtain the required result.
\end{proof}

\begin{theorem}\label{Turan density of projective hyperplane}
Each $(q+1)-$graph $H$ on $n$ vertices that does not contain a copy of $\PG(m,q)$ has at most $\left(1-\frac{1}{(q+1)|E(\PG(m-1,q))|}\right)^{\frac{1}{1+q^{m-1}}}\frac{n^{q+1}}{(q+1)!}$ edges, i.e.
\begin{equation*}
\mathrm{ex}(n,\PG(m,q))\leq\left(1-\frac{1}{(q+1)|E(\PG(m-1,q))|}\right)^{\frac{1}{1+q^{m-1}}}\frac{n^{q+1}}{(q+1)!}.
\end{equation*}
Moreover, the Tur\'{a}n density 
\begin{equation*}
\pi(\PG(m,q))\leq\left(1-\frac{1}{(q+1)|E(\PG(m-1,q))|}\right)^{\frac{1}{1+q^{m-1}}}=\left(1-\frac{(q-1)^{2}}{(q^{m}-1)(q^{m-1}-1)}\right)^{\frac{1}{1+q^{m-1}}}.    
\end{equation*}
\end{theorem}

\begin{proof}
Let $H$ be $\PG(m,q)-$free $(q+1)-$graph on $n$ vertices, $v_1,\cdots,v_{|V(\PG(m-1,q))|}\in V(H)$ and $\{x_a:a\in\mathbb{F}_q^m\}\subseteq V(H)$. Recall that $S(x_a:a\in\mathbb{F}_q^m)$ denotes the set of parallel classes $P_{x}(x_a:a\in\mathbb{F}_q^m)$ with direction $x$ of a copy of $\PG(m,q)$ with 
\begin{equation*}
 V(\PG(m,q))=\{x_a:a\in\mathbb{F}_q^m\}\sqcup\{v_1,\cdots,v_{|V(\PG(m-1,q))|}\}   
\end{equation*}
whose directions are in $\{v_1,\cdots,v_{|V(\PG(m-1,q))|}\}$. Then the local forbidden constraint is
\begin{align}\label{local PG(m,q)}
\underset{\substack{E\in E(\PG(m-1,q))\\V(\PG(m-1,q))=\{v_1,\cdots,v_{|V(\PG(m-1,q))|}\}}}{\sum}\mathds{1}_{E(H)}(E)&\left[\underset{x\in E}{\sum}\ \underset{A\in P_{x}(x_a:a\in\mathbb{F}_q^m)}{\prod}\mathds{1}_{E(H)}(x\sqcup\{A\})\right]\notag\\
&\leq|E(\PG(m-1,q))|(q+1)-1.
\end{align}
Hence, the global forbidden constraint is
\begin{align}\label{global PG(m,q)}
\underset{\substack{\{v_1,\cdots,v_{|V(\PG(m-1,q))|}\},\\\{x_a:a\in\mathbb{F}_q^m\}\subseteq V(H)}}{\sum}\underset{\substack{E\in E(\PG(m-1,q))\\V(\PG(m-1,q))=\\\{v_1,\cdots,v_{|V(\PG(m-1,q))|}\}}}{\sum}&\mathds{1}_{E(H)}(E)
\left[\underset{x\in E}{\sum}\ \underset{A\in P_{x}(x_a:a\in\mathbb{F}_q^m)}{\prod}\mathds{1}_{E(H)}(x\sqcup\{A\})\right]\notag\\
&\leq\left(|E(\PG(m-1,q))|(q+1)-1\right)n^{|V(\PG(m,q))|}.
\end{align}

Now we calculate a lower bound for the left hand side term in \eqref{global PG(m,q)} using H\"{o}lder's inequality. To do this, we calculate a single term for one edge $E\in E(\PG(m-1,q))$ and one $x\in E$. Without loss of generality, we assume $E=\{v_1,\cdots,v_{q+1}\}$ and $x=v_1$. The calculation is similar for the other terms.
\begin{align*}
&\underset{\substack{\{v_1,\cdots,v_{|V(\PG(m-1,q))|}\},\\\{x_a:a\in\mathbb{F}_q^m\}\subseteq V(H)}}{\sum}\mathds{1}_{E(H)}(\{v_1,\cdots,v_{q+1}\})
\left[\underset{A\in P_{v_1}(x_a:a\in\mathbb{F}_q^m)}{\prod}\mathds{1}_{E(H)}(\{v_1\}\sqcup\{A\})\right]\\
=&n^{|V(\PG(m-1,q))|-(q+1)}\underset{\{v_1,\cdots,v_{q+1}\}\subseteq V(H)}{\sum}\mathds{1}_{E(H)}(\{v_1,\cdots,v_{q+1}\})\\
&\hspace{7cm}\underset{\{x_a:a\in\mathbb{F}_q^m\}\subseteq V(H)}{\sum}\left[\underset{A\in P_{v_1}(x_a:a\in\mathbb{F}_q^m)}{\prod}\mathds{1}_{E(H)}(\{v_1\}\sqcup\{A\})\right]\\
=&n^{|V(\PG(m-1,q))|-(q+1)}\underset{\{v_1,\cdots,v_{q+1}\}\subseteq V(H)}{\sum}\mathds{1}_{E(H)}(\{v_1,\cdots,v_{q+1}\})\\
&\hspace{7cm}\left[\underset{x_1,\cdots,x_q\in V(H)}{\sum}\mathds{1}_{E(H)}(\{v_1,x_1,\cdots,x_q\})\right]^{|P_{v_1}|}\\
=&n^{|V(\PG(m-1,q))|-(q+1)}\underset{v_1\in V(H)}{\sum}\left[\underset{v_2,\cdots,v_{q+1}\in V(H)}{\sum}\mathds{1}_{E(H)}(\{v_1,\cdots,v_{q+1}\})\right]^{|P_{v_1}|+1}\\
\geq&n^{|V(\PG(m-1,q))|-(q+1)}\frac{1}{n^{|P_{v_1}|}}\left[\underset{v_1,\cdots,v_{q+1}\in V(H)}{\sum}\mathds{1}_{E(H)}(\{v_1,\cdots,v_{q+1}\})\right]^{|P_{v_1}|+1}\\
=&n^{|V(\PG(m-1,q))|-(q+1)-q^{m-1}}\left((q+1)!|E(H)|\right)^{q^{m-1}+1}\\
\end{align*}
Hence,
\begin{align*}
\underset{\substack{\{v_1,\cdots,v_{|V(\PG(m-1,q))|}\},\\\{x_a:a\in\mathbb{F}_q^m\}\subseteq V(H)}}{\sum}&\underset{\substack{E\in E(\PG(m-1,q))\\V(\PG(m-1,q))=\\\{v_1,\cdots,v_{|V(\PG(m-1,q))|}\}}}{\sum}\mathds{1}_{E(H)}(E)
\left[\underset{x\in E}{\sum}\ \underset{\substack{A\in P(x_a:a\in\mathbb{F}_q^m)\in S(x_a:a\in\mathbb{F}_q^m)\\\text{dir}P(x_a:a\in\mathbb{F}_q^m)=x}}{\prod}\mathds{1}_{E(H)}(x\sqcup\{A\})\right]\\
\geq&|E(\PG(m-1,q))|(q+1)n^{|V(\PG(m-1,q))|-(q+1)-q^{m-1}}\left((q+1)!|E(H)|\right)^{q^{m-1}+1}.
\end{align*}

Therefore,
\begin{tcolorbox}
\begin{align*}
|E(\PG(m-1,q))|&(q+1)n^{|V(\PG(m-1,q))|-(q+1)-q^{m-1}}\left((q+1)!|E(H)|\right)^{q^{m-1}+1}\\
&\leq\left(|E(\PG(m-1,q))|(q+1)-1\right)n^{|V(\PG(m,q))|},
\end{align*}
\end{tcolorbox}
solving which we obtain the required result.
\end{proof}

\section{Further results: Tur\'an density estimation of \texorpdfstring{$B_{t}[\PG(2,q)]$}{}}
We conclude our article with estimating the Tur\'{a}n density of book graph $B_t\left[\PG(2,q)\right]$. The formal construction and its definition are as follows. 
\begin{construction}
Let $G_{1},\ldots,G_{t}$ be $t-$many vertex disjoint copies of $\PG(2,q)$. For each $i\in[t]$ let $\phi_i:\mathbb{K}_{q+1}^{(q+1)}\hookrightarrow G_{i}$ be an embedding. We define an equivalence relation $\sim$ on $V(G_{1})\sqcup\cdots\sqcup V(G_{t})$ as follows: for $x,y\in V(G_{1})\sqcup\cdots\sqcup V(G_{t})$, we have $x\sim y$ if and only if either $x=y$ or there exists $v\in V(\mathbb{K}_{q+1}^{(q+1)})$ such that $x=\phi_i(v)$ and $y=\phi_j(v)$ for some $i,j\in[t]$.

We consider the quotient map  $q(x)=[x]$, for each $x\in V(G_{1})\sqcup\cdots\sqcup V(G_{t})$. Here $[x]$ denotes the equivalent class containing $x$.
\end{construction}

\begin{definition}
The \emph{book hypergraph} $B_t\left[\PG(2,q)\right]$ with pages $\PG(2,q)$ and \emph{spine} $\mathbb{K}_{q+1}^{(q+1)}$ consists of the following vertex set and edge set.
\begin{align*}
V(B_t\left[\PG(2,q)\right])&=\left(V(G_{1})\sqcup\cdots\sqcup V(G_{t})\right)/{\sim},\\
E(B_t\left[\PG(2,q)\right])&=\underset{i\in[t]}{\bigcup}\bigg\{Q(e):e\in E(G_i)\bigg\},
\end{align*}
where $Q(e):=\{Q(v):v\in e\}$.      
\end{definition}

\begin{example}
$B_{2}[\PG(2,2)]$ with spine $\mathbb{K}_{3}^{(3)}$ equals the following $3-$graph:-
\begin{align*}
\left\{\{x,y,z\},
\begin{aligned}
&\{x,a,b\},\{x,c,d\}, \{x,a',b'\}, \{x,c',d'\},\\
&\{y,a,c\},\{y,b,d\},\{y,a',c'\},\{y,b',d'\},\\
&\{z,a,d\},\{z,b,c\},\{z,a',d'\},\{z,b',c'\}    
\end{aligned}
\right\}.    
\end{align*}
\end{example}

In the following, we estimate the Tur\'{a}n density of $B_{t}[\PG(2,q)]$ for $t\geq2$, using the $\mathsf{L}^{p}-$method.

\begin{theorem}\label{Turan density of projective book}
For each $(q+1)-$graph $H$ on $n$ vertices that does not contain $t-$page book with pages $\PG(2,q)$, contains at most  $\left(\frac{q}{q+1}\right)^{\frac{1}{qt}}\frac{n^{q+1}}{(q+1)!}$ edges, i.e 
\begin{equation*}
\mathrm{ex}(n,B_{t}[\PG(2,q)])\leq\left(\frac{q}{q+1}\right)^{\frac{1}{qt}}\frac{n^{q+1}}{(q+1)!}.   
\end{equation*}
Moreover, $\pi(B_{t}[\PG(2,q)])\leq\left(\frac{q}{q+1}\right)^{\frac{1}{qt}}$.
\end{theorem}

\begin{proof}
Here we use the argument analogous to that in the proof Theorem~\ref{Turan density of projective plane}. For a $B_{t}[\PG(2,q)]$-free $(q+1)-$graph $H$ on $n-$vertices, we have the following global forbidden constraint:-
\begin{equation}\label{global B_{t}[PG(2,q)]}
\underset{\{y_1,\cdots,y_{q+1}\}\in E(H)}{\sum}\underset{\{x_a^{(k)}:a\in\mathbb{F}_q^2,k\in[t]\}\subseteq V(H)}{\sum}\underset{m\in[q+1]}{\sum}\underset{k\in[t]}{\prod}\underset{e\in P_m(x_a^{(k)}:a\in\mathbb{F}_q^2)}{\prod}\mathds{1}_{E(H)}(\{y_{m}\}\sqcup e)\leq qn^{q^2t}|E(H)|.
\end{equation}

To better understand the expression~\ref{global B_{t}[PG(2,q)]}, we present the following illustrative example, for a $B_{2}[\PG(2,2)]$-free $3-$graph $H$ on $n-$vertices.
\begin{align*}
&\underset{\{x,y,z\}\in E(H)}{\sum}\underset{a,a',b,b',c,c',d,d'\in V(H)}{\sum}\bigg(\mathds{1}_{E(H)}(\{x,a,b\})\mathds{1}_{E(H)}(\{x,c,d\})\mathds{1}_{E(H)}(\{x,a',b'\})\mathds{1}_{E(H)}(\{x,c',d'\})+\\
&\hspace{2cm}+\mathds{1}_{E(H)}(\{y,a,c\})\mathds{1}_{E(H)}(\{y,b,d\})\mathds{1}_{E(H)}(\{y,a',c'\})\mathds{1}_{E(H)}(\{y,b',d'\})+\\
&\hspace{2cm}+\mathds{1}_{E(H)}(\{z,a,d\})\mathds{1}_{E(H)}(\{z,b,c\})\mathds{1}_{E(H)}(\{z,a',d'\})\mathds{1}_{E(H)}(\{z,b',c'\})\bigg)\leq2n^{8}|E(H)|.    
\end{align*}

Using similar arguments and applying H\"{o}lder's inequality at the appropriate steps, we obtain the lower bound $\frac{1}{(q+1)!}\frac{q+1}{n^{qt}}\left((q+1)!|E(H)|\right)^{qt+1}$ for the left hand side expression to~\ref{global B_{t}[PG(2,q)]}. Hence,
\begin{tcolorbox}
\begin{align*}
\frac{1}{(q+1)!}\frac{q+1}{n^{qt}}\left((q+1)!|E(H)|\right)^{qt+1}\leq qn^{q^2t}|E(H)|
\end{align*}
\end{tcolorbox}
solving which we obtain the required result.
\end{proof}

\begin{remark}
One can construct the $t-$page book where the spine is $\PG(m-1,q)$ and the pages are $\PG(m,q)$ and obtain similar results.
\end{remark}

\begin{acknowledgement}
The author Subhankar Dash acknowledges the National Institute of Science Education and Research (NISER), Bhubaneswar and the Homi Bhabha National Institute (HBNI), Mumbai, for financial support through a doctoral fellowship. This research work has no associated data.

The authors used an artificial intelligence (AI) based language assistant to improve the clarity and presentation of the manuscript, including suggestions on grammar, style and exposition. The mathematical results, proofs and conclusions are entirely the work of the authors, who take full responsibility for the content of the manuscript.
\end{acknowledgement}

\bibliographystyle{amsplain}

\end{document}